\documentclass[a4paper,12pt]{article}
\usepackage[cp1251]{inputenc}
\usepackage[russian, ukrainian,english]{babel}
\usepackage{amsmath, amsthm, amsfonts, amssymb}

\theoremstyle{plain}
\newtheorem{thm}{Theorem}[section]

\newtheorem{lem}{Lemma}[section]

\newtheorem{rem}{Remark}[section]

\begin{document}

\begin{center}
{\bf \Large A representation for the Kantorovich--Rubinstein distance on the abstract Wiener space}

\vskip20pt

G. V. Riabov
\footnote{The author is grateful to Prof. A. A. Dorogovtsev and A. M. Kulik for valuable discussions and comments. The research is partially supported by the Young Scientists Grant of the National Academy of Sciences of Ukraine.}

{\it Institute of Mathematics, NAS of Ukraine}

\end{center}

\begin{abstract}
A representation for the Kantorovich--Rubinstein distance between probability measures on an abstract Wiener space in terms of the extended stochastic integral (or, divergence) operator is obtained.
\end{abstract}

\section{Introduction}

Consider the abstract Wiener space $(X, H, \mu).$ That is, $(X,\|\cdot\|)$ is a separable Banach space; $\mu$ is a centered Gaussian measure on the Borel $\sigma-$field of $X,$ such that $\mbox{supp } \mu=X;$ $(H,|\cdot|)$ is a separable Hilbert space, that is densely and continuously embedded in $X$ and is such that 
$$
\int_{X} \exp(il(x))\mu(dx)=\exp\bigg(-\frac{1}{2}|l|^2\bigg), \ l\in X^*.
$$
  
The space $\mathcal{M}(X)$ of Borel probability measures on $X$ is endowed with a Kantorovich-Rubinstein distance \cite[\S 1.2]{Villani}
$$
W_1(\nu_1,\nu_2)=\inf_{\pi\in C(\nu_1,\nu_2)}\int_X\int_X |x_1-x_2|\pi(dx_1,dx_2),
$$
where $C(\nu_1,\nu_2)$ is the set of all Borel probability measures on $X\times X$ with marginals $\nu_1$ and $\nu_2.$ 

The aim of the present paper is to establish the following representation for $W_1.$ 

\begin{thm}
\label{thm_main}
Consider probability measures $\nu_0,\nu_1\in \mathcal{M}(X)$ with $\nu_1-\nu_0\ll \mu$ and $\frac{d(\nu_1-\nu_0)}{d\mu}\in L^2(X,\mu).$ Then 
\begin{equation}
\label{eq_main}
W_1(\nu_0,\nu_1)=\inf_{Iu=\frac{d(\nu_1-\nu_0)}{d\mu}}\bigg\{\int_X |u(x)|\mu(dx) \bigg\}.
\end{equation}
\end{thm}

In \eqref{eq_main} $I$ denotes the extended stochastic integral (or divergence operator, see the next section for precise definitions), and the infimum is taken over all vector fields $u:X\to H$ that solve the equation
\begin{equation}
\label{eq_eq}
Iu=\frac{d(\nu_1-\nu_0)}{d\mu}.
\end{equation}

This work was partially motivated by results of \cite{DIRS}  where several  integral representations for functionals from the Gaussian white noise were derived. Namely, for every random variable $\alpha\in L^2(W,\mu)$ the equation
\begin{equation}
\label{int_representation}
\alpha=\int_X \alpha d\mu + Iu
\end{equation}
has infinitely many solutions $u:X\to H.$  In the case of the classical Wiener space with $X=C_0([0,1])$ and $\mu$ being the Wiener measure, there is unique solution $u_0$ of \eqref{int_representation} that is adapted to the natural filtration \cite[Ch. V, \S 3]{RevuzYor}. When $\alpha$ is the probability density, i.e. $\alpha\geq0,$ $\int_X \alpha d\mu=1,$ the representation
$$
\alpha=1 + Iu_0
$$
is connected to the measure transportation via the Girsanov theorem \cite[Ch. VIII, \S 1]{RevuzYor}: the mapping 
$$
T(x)=x(t)-\int^t_0 \frac{u_0(s,x)}{1+I(u_0 1_{\cdot\leq s})(x)}ds
$$
sends the measure $\alpha\cdot \mu$ into the Wiener measure $\mu,$
$$
(\alpha\cdot \mu)\circ T^{-1}=\mu.
$$
Moreover, the mapping $T$ is in a sense optimal \cite{Lehec, Follmer}: for every mapping $S:X\to X,$ such that $S(x)-x\in H$ and $(\alpha\cdot \mu)\circ S^{-1}=\mu,$ one has 
$$
\int_X |T(x)-x|^2\mu(dx)\leq \int_X |S(x)-x|^2\mu(dx).
$$
When $(X,H,\mu)$ is a general abstract Wiener space there is still a connection between measure transport and the equation \eqref{int_representation}. One result in this direction was obtained in \cite{FeyelUstunel}.  It was proved that for sufficiently smooth density $\alpha$ one has
\begin{equation}
\label{FU}
W_1(\alpha\cdot \mu,\mu)\leq \int_X |(1+L)^{-1}D\alpha| d\mu,
\end{equation}
where $D$ denotes the stochastic derivative and $(-L)$ is the generator of the Ornstein-Uhlenbeck semigroup. Our result \eqref{eq_main} generalizes this inequality. Indeed, the identity \cite[Remark 5.8.7]{Bogachev}
$$
ID=L
$$
implies that $(1+L)^{-1}D\alpha$ is a solution to \eqref{int_representation}.

Another motivation for the undertaken research is the study of geodesics on the space $(\mathcal{M}(X),W_1)$ \cite[Ch. 7]{AGS}. In the case $p>1$ the differential structure of spaces $(\mathcal{M}(X),W_p)$ is studied rather detaily and with a number of applications to functional inequalitites \cite{vRS, FSS, N, FN}. The assumption $p>1$ allows to apply powerful technique from convex analysis. In the limit $p\to 1+$ certain results about geodesics in $(\mathcal{M}(X),W_1)$ then can be obtained \cite{N}. However, the distance $W_1$ is not strictly convex.  This results in existence of multiple geodesics between different measures, while the described approximating approach gives results only for  particular $W_1$-geodesics. In general, the behaviour of geodesics in the space $(\mathcal{M}(X),W_1)$ remains unstudied. Proved identity \eqref{eq_main} gives an intrinsic description of the $W_1$-distance between measures. In our further work it will be applied to the study of $W_1-$geodesics between measures on an abstract Wiener space.

\section{Notations and Preliminary Results}

For a detailed exposition of the theory of Gaussian measures on Banach spaces we refer to \cite{Bogachev}. 

A function $f:X\to \mathbb{R}$ will be called a smooth cylindrical function, if it has a representation 
$$
f(x)=\varphi(l_1(x),\ldots,l_d(x)), \ x\in X, 
$$
where $l_1,\ldots,l_d\in X^*,$ $\varphi:\mathbb{R}^d\to \mathbb{R}$ is infinitely differentiable function bounded together with all derivatives.  Denote by $\mathcal{FC}^\infty(X)$ the family of all smooth cylindrical functions. In the finite dimensional case, $\mathcal{FC}^\infty(X)$ coincides with the family of all infinitely differentiable functions bounded together with all derivatives.

Stochastic derivative $D$ is naturally defined for a fuction $f\in \mathcal{FC}^\infty$ with a representation $f(x)=\varphi(l_1(x),\ldots,l_d(x))$:

$$
Df(x)=\sum^d_{i=1}\partial_i \varphi(l_1(x),\ldots,l_d(x)) l_i\in H.
$$

Then $D$ is extended to a closed (unbounded) operator 
$$
D:L^2(X,\mu)\to L^2(X,\mu;H).
$$

Denote by $I$ the adjoint operator to $D,$
$$
I=D^*.
$$
Following \cite{Skorokhod} we will call $I$ the extended stochastic integral. In terms of the integration by parts formula one has the following: for all $f\in \mathcal{FC}^\infty$
$$
\int_X (u,Df) d\mu=\int_X Iu \cdot f d\mu
$$
 \cite[\S 5.8]{Bogachev}.

\begin{rem}
\label{rem1}
In \cite{Bogachev} the operator $(-I)$ is denoted by $\delta$ and is called a divergence operator, while the term ``extended stochastic integral'' is kept for a specific situation when $H$ is an $L^2$-space. Our terminology is chosen to underline the connection between the operator $I$ and integral representations of random variables \eqref{int_representation}.
\end{rem}

The Ornstein-Uhlenbeck semigroup is denoted by $(T_t)_{t\geq 0}:$
$$
T_t h(x)=\int_X h(e^{-t}x+\sqrt{1-e^{-2t}}y)\mu(dy).
$$
For each $p\geq 1$ $(T_t)_{t\geq0}$ is a strongly continuous semigroup of contractions in $L^p(X,\mu)$ \cite{Bogachev}. We will also consider the action of $T_t$ on measures. Given a signed measure $\nu$ on $X,$ define
$$
T_t \nu(A) =\int_X T_t 1_A(x) \nu(dx)=\int_X\int_X 1_A(e^{-t}x+\sqrt{1-e^{-2t}}y)\mu(dy) \nu(dx).
$$
Duality considerations imply that $T_t$ is still a contraction:
$$
\|T_t \nu\|_v\leq \|\nu\|_v,
$$ 
where $\|\cdot\|_v$ denotes the total variation norm.

Among integral representations \eqref{int_representation} of a random variable $\alpha$ there is a unique representation with a minimal $L^2(X,\mu;H)$-norm \cite{DIRS}. In the next lemma the needed properties of this representation are gathered.

\begin{lem} 
\label{lem_min_norm} \cite[L. 6,7]{DIRS}. Define the mapping
$$
v(\alpha)=D \int^\infty_0 T_t \alpha dt, \ \alpha\in L^2(X,\mu).
$$
Then 

\begin{itemize}

\item $v:L^2(X,\mu)\to L^2(X,\mu;H)$ is  a bounded linear operator of norm $1;$

\item for every $\alpha\in L^2(X,\mu),$  $v(\alpha)$ is a solution to \eqref{int_representation}:
$$
\alpha=\int_X \alpha d\mu + Iv(\alpha);
$$
 
\item for any solution $u$ to \eqref{int_representation}, one has
$$
\int_X |v(\alpha)|^2 d\mu\leq \int_X |u|^2 d\mu.
$$ 

\end{itemize}

\end{lem}

\section{Finite Dimensional Case}

In this section we prove a partial case of the theorem \ref{thm_main}. Assume that $X$ is a finite dimensional space, $X=\mathbb{R}^n,$ and that $\mu$ is a standard Gaussian measure. In this case $H=\mathbb{R}^n$ and $|\cdot|$ is the Euclidean norm.

\begin{lem}
\label{lem_finite_dim}
For Borel probability measures $\nu_0,\nu_1$  on $\mathbb{R}^n,$ such that
$$
\nu_1-\nu_0\ll \mu, \ \frac{d(\nu_1-\nu_0)}{d\mu}\in L^2(\mathbb{R}^n,\mu),
$$
the representation \eqref{eq_main} holds:
$$
W_1(\nu_0,\nu_1)=\inf_{Iu=\frac{d(\nu_1-\nu_0)}{d\mu}}\bigg\{\int_{\mathbb{R}^n} |u(x)|\mu(dx) \bigg\}.
$$

\end{lem}

\begin{proof}

{\it Step 1.} Assume a stronger condition of absolute continuity: for $i=0,1$  
$$
\nu_i\ll\mu, \ \inf\frac{d\nu_i}{d\mu}\geq \varepsilon >0.
$$

The well-known Kantorovich--Rubinstein theorem \cite[Th. 1.14]{Villani} states that  
$$
W_1(\nu_0,\nu_1)=\sup\bigg\{\int_{\mathbb{R}^n} f d(\nu_1-\nu_0)\bigg\},
$$
where the supremum is taken  over all bounded $1-$Lipschitz functions $f:\mathbb{R}^n\to\mathbb{R}.$  
Hence, to prove $\leq$ in the representation \eqref{eq_main}, it is enough to check the inequality
\begin{equation}
\label{eq_1}
\bigg|\int_{\mathbb{R}^n} f d\nu_0 -\int_{\mathbb{R}^n} f d\nu_1 \bigg|\leq \int_{\mathbb{R}^n} |u| d\mu,
\end{equation}
where $f:\mathbb{R}^n\to\mathbb{R}$  is a bounded $1-$Lipschitz functions  and $u\in  L^2(\mathbb{R}^n,\mu; \mathbb{R}^n)$ satisfies \eqref{eq_eq}:
$$
Iu=\frac{d(\nu_1-\nu_0)}{d\mu}.
$$
At first we prove \eqref{eq_1} under additional smoothness assumption on $f.$

\begin{lem}
\label{lemma_finite_dim}
Inequality \eqref{eq_1} holds for all bounded twice continuously differentiable functions $f:\mathbb{R}^n\to\mathbb{R},$ such that
$$
\sup_{x\in \mathbb{R}^n}
|Df(x)|\leq 1, \ \sup_{x\in \mathbb{R}^n}|D^2f(x)|<\infty.
$$
\end{lem}

\begin{proof}

Consider the flow of measures 
$$
\nu_t=\nu_0+t(\nu_1-\nu_0), \ 0\leq t\leq 1.
$$
Assumptions of the Step 1 imply that all measures $\nu_t$ are absolutely continuous with respect to $\mu.$ Denote by $\alpha_t$ their densities: $\alpha_t=\frac{d\nu_t}{d\mu}.$ Again, by assumptions, 
$$
\alpha_t \geq \varepsilon.
$$
In terms of densities $\alpha_t,$ the vector field $u$ is a solution of
$$
Iu=\alpha_1-\alpha_0.
$$
Hence, for $s<t$
\begin{equation}
\label{eq_2}
\alpha_t=\alpha_s+(t-s)Iu.
\end{equation}
Next we define transformations of $\mathbb{R}^n$ that allow to control the distance $W_1(\nu_{\frac{k}{m}},\nu_{\frac{k+1}{m}}).$ Let $m\geq 1$ be fixed. For every $k=0,1,\ldots,m-1$ define the mapping
$$
\varphi_k(x)=x+\frac{u(x)}{m\alpha_{\frac{k}{m}}(x)}, \ x\in\mathbb{R}^n.
$$
To compare integrals of $f$ with respect to $\nu_{\frac{k+1}{m}}$ and $\nu_{\frac{k}{m}}\circ \varphi^{-1}_k$ we use first order Taylor approximation:
$$
f(\varphi_k(x))=f(x)+\frac{(Df(x),u(x))}{m\alpha_{\frac{k}{m}}(x)}+r(x), 
$$
where the remainder $r(x)$ satisfies
$$
|r(x)|\leq C \frac{|u(x)|^2}{m^2 \alpha^2_{\frac{k}{m}}(x)}\leq C \frac{|u(x)|^2}{m^2 \varepsilon \alpha_{\frac{k}{m}}(x)}.
$$
Using the relation \eqref{eq_2} and the integration by parts, we can compare integrals:
\begin{equation}
\label{eq_3}
\begin{gathered}
\bigg|\int_{\mathbb{R}^n} f(\varphi_k(x)) \nu_{\frac{k}{m}}(dx) -\int_{\mathbb{R}^n} f(x) \nu_{\frac{k+1}{m}}(dx) \bigg|= \\
=\bigg|\int_{\mathbb{R}^n}\bigg( f(x)+\frac{(Df(x),u(x))}{m\alpha_{\frac{k}{m}}(x)}+r(x)\bigg) \alpha_{\frac{k}{m}}(x)\mu(dx) -\int_{\mathbb{R}^n} f(x) \bigg(\alpha_{\frac{k}{m}}(x)+\frac{1}{m}Iu(x)\bigg)\mu(dx) \bigg|\leq \\
\leq \frac{C}{m^2 \varepsilon}\int_{\mathbb{R}^n}|u|^2 d\mu.
\end{gathered}
\end{equation}

The comparison of  $\nu_{\frac{k}{m}}$ and $\nu_{\frac{k}{m}}\circ \varphi^{-1}_k$ is simpler.
\begin{equation}
\label{eq_4}
\begin{gathered}
\bigg|\int_{\mathbb{R}^n} f(\varphi_k(x)) \nu_{\frac{k}{m}}(dx) -\int_{\mathbb{R}^n} f(x) \nu_{\frac{k}{m}}(dx) \bigg|\leq \\
\leq \int_{\mathbb{R}^n} \bigg|f\bigg(x+\frac{u(x)}{m\alpha_{\frac{k}{m}}(x)}\bigg)-f(x)\bigg| \alpha_{\frac{k}{m}}(x)\mu(dx)\leq\frac{1}{m} \int_{\mathbb{R}^n} |u|d\mu.
\end{gathered}
\end{equation}

Combine \eqref{eq_3} and \eqref{eq_4}.
$$
\bigg|\int_{\mathbb{R}^n} f d\nu_0 -\int_{\mathbb{R}^n} f d\nu_1 \bigg|\leq 
$$
$$
\leq \sum^{m-1}_{k=0} \bigg(\bigg|\int_{\mathbb{R}^n} f(\varphi_k(x)) \nu_{\frac{k}{m}}(dx) -\int_{\mathbb{R}^n} f(x) \nu_{\frac{k}{m}}(dx) \bigg|+
$$
$$
+\bigg|\int_{\mathbb{R}^n} f(\varphi_k(x)) \nu_{\frac{k}{m}}(dx) -\int_{\mathbb{R}^n} f(x) \nu_{\frac{k+1}{m}}(dx) \bigg|\bigg)\leq
$$
$$
\leq  \int_{\mathbb{R}^n} |u|d\mu + \frac{C}{m \varepsilon}\int_{\mathbb{R}^n}|u|^2 d\mu.
$$
It remains to let $m\to \infty.$ The lemma is proved.

\end{proof}

\begin{rem} In \cite{FeyelUstunel} densities $\alpha_t$ are assumed to be smooth enough for a flow of solutions of a differential equation
$$
d\varphi_t(x)=\frac{(1+L)^{-1}D(\alpha_1-\alpha_0)(x)}{\alpha_t(x)}dt
$$ 
to exist. As shown in the proof of the lemma \ref{lemma_finite_dim} only the discretized version of such flow is needed to handle the case with rather general densities and arbitrary $u$ at the place of $(1+L)^{-1}D(\alpha_1-\alpha_0).$

\end{rem}

Let $f:\mathbb{R}^n\to\mathbb{R}$ be bounded $1-$Lipschitz function. Consider an application of the Ornstein-Uhlenbeck semigroup to $f:$
$$
f_t(x)=T_tf(x).
$$
Then for each $t>0$ $f_t$ satisfies conditions of the lemma \ref{lemma_finite_dim}. The strong coninuity of $(T_t)$ implies
$$
\int_{\mathbb{R}^n} (f_t-f)^2d\mu \to 0, \ t\to 0.
$$
Using the assumption $\frac{d(\nu_1-\nu_0)}{d\mu}\in L^2(\mathbb{R}^n,\mu)$ we get the following.
$$
\bigg|\int_{\mathbb{R}^n} f d\nu_0 -\int_{\mathbb{R}^n} f d\nu_1 \bigg|=\bigg|\int_{\mathbb{R}^n} f \frac{d(\nu_1-\nu_0)}{d\mu} d\mu \bigg|=
$$
$$
=\lim_{t\to 0}\bigg|\int_{\mathbb{R}^n} f_t d\nu_0 -\int_{\mathbb{R}^n} f_t d\nu_1 \bigg|\leq \int_{\mathbb{R}^n} |u| d\mu.
$$

For all measures $\nu_0,\nu_1\ll \mu,$ $\inf \frac{d\nu_0}{d\mu}>0,\inf \frac{d\nu_1}{d\mu}>0$ and all solutions $u$ of \eqref{eq_eq}, the inequality 
$$
W_1(\nu_0,\nu_1)\leq \int_{\mathbb{R}^n}|u|d\mu
$$
is proved.

{\it Step 2.} The case $\nu_0,\nu_1\ll \mu$ reduces to the previous one by transformation
$$
\nu_{i,\varepsilon}=\frac{\nu_i+\varepsilon\mu}{1+\varepsilon}, \ i=0,1.
$$
Indeed, for any $u,$ such that  
$$
Iu=\frac{d(\nu_1-\nu_0)}{d\mu},
$$
one has 
$$
I\bigg(\frac{u}{1+\varepsilon}\bigg)=\frac{d(\nu_{1,\varepsilon}-\nu_{0,\varepsilon})}{d\mu},
$$
and, by the result of step 1,
$$
W_1(\nu_0,\nu_1)=(1+\varepsilon)W_1(\nu_{0,\varepsilon},\nu_{1,\varepsilon})\leq(1+\varepsilon)\int_{\mathbb{R}^n}\frac{|u|}{1+\varepsilon}d\mu=\int_{\mathbb{R}^n}|u|d\mu.
$$

{\it Step 3.} We prove inequality $\leq$ in \eqref{eq_main} without any additional assumptions. Consider probability measures $\nu_0,\nu_1$ on $\mathbb{R}^n,$ such that 
$$
\nu_1-\nu_0\ll \mu, \ \frac{d(\nu_1-\nu_0)}{d\mu}\in L^2(\mathbb{R}^n,\mu).
$$
Define measures 
$$
\nu_{i,t}=T_t\nu_i, \ i=0,1.
$$
Then $\nu_{i,t}\ll \mu$ with the density
$$
\frac{d\nu_{i,t}}{d\mu}(x)=(1-e^{-2t})^{-\frac{n}{2}}\int_{\mathbb{R}^n}e^{-\frac{|x|^2-2e^t(z,x)+|z|^2}{2(e^{2t}-1)}}\nu_i(dz).
$$
Symmetry of $T_t$ in $L^2(X,\mu)$ implies the relation
$$
\frac{d(T_t\nu_1-T_t\nu_0)}{d\mu}=T_t\frac{d(\nu_1-\nu_0)}{d\mu}.
$$
Hence, for each bounded $1-$Lipschitz continuous function $f:\mathbb{R}^n\to\mathbb{R}$ one has
$$
\begin{gathered}
\bigg|\int_{\mathbb{R}^n} f (d\nu_1-d\nu_0) - \int_{\mathbb{R}^n} f (d\nu_{1,t}-d\nu_{0,t}) \bigg|=\bigg|\int_{\mathbb{R}^n} f\bigg(\frac{d(\nu_1-\nu_0)}{d\mu}-T_t\frac{d(\nu_1-\nu_0)}{d\mu}\bigg) d\mu\bigg|\leq \\
\leq \sqrt{n\int_{\mathbb{R}^n}\bigg(\frac{d(\nu_1-\nu_0)}{d\mu}-T_t\frac{d(\nu_1-\nu_0)}{d\mu}\bigg)^2 d\mu}.
\end{gathered}
$$
In particular,
\begin{equation}
\label{eq_l2_ineq}
|W_1(\nu_{0,t},\nu_{1,t})- W_1(\nu_{0},\nu_{1})| \leq \sqrt{n\int_{\mathbb{R}^n}\bigg(\frac{d(\nu_1-\nu_0)}{d\mu}-T_t\frac{d(\nu_1-\nu_0)}{d\mu}\bigg)^2 d\mu}.
\end{equation}
Given a solution $u$ of \eqref{eq_eq} it is easily verified that $e^{-t}T_t u$ satisfies
$$
I(e^{-t} T_t u)=\frac{d(\nu_{1,t}-\nu_{0,t})}{d\mu}.
$$
Hence, by the result of the step 2
$$
W_1(\nu_{0},\nu_{1})=\lim_{t\to 0}W_1(\nu_{0,t},\nu_{1,t})\leq \lim_{t\to 0} \int_{\mathbb{R}^n}e^{-t} |T_t u| d\mu=\int_{\mathbb{R}^n}|u| d\mu.
$$
The inequality $\leq$ is proved.

{\it Step 4.} At this step we observe that both sides of \eqref{eq_main} are continuous functions of the density  $\alpha=\frac{d(\nu_1-\nu_0)}{d\mu}$ (when the distance between densities is the $L^2(\mathbb{R}^n,\mu)$-distance). For the left-hand side it was checked at the step 3 (see \eqref{eq_l2_ineq}). For the right-hand side this follows from the existence of the minimal norm representation operator $v$ (see lemma \ref{lem_min_norm}). Denote the right-hand side of \eqref{eq_main} by $\mathcal{N}:$
$$
\mathcal{N}(\alpha)=\inf_{Iu=\alpha}\bigg\{\int_{\mathbb{R}^n} |u| d\mu \bigg\}.
$$
Consider two densitites
$$
\alpha, \beta \in L^2(\mathbb{R}^n,\mu), \int \alpha d\mu=\int\beta d\mu=0.
$$
For each solution $u$ of $Iu=\alpha,$ one has
$$
I(u+v(\beta-\alpha))=\beta.
$$
Hence, by the properties of $v,$
$$
\mathcal{N}(\beta)\leq \int_{\mathbb{R}^n} |u+v(\beta-\alpha)| d\mu \leq \int_{\mathbb{R}^n} |u|d\mu +\sqrt{\int_{\mathbb{R}^n} |v(\beta-\alpha)|^2 d\mu}
$$
$$
\leq\int_{\mathbb{R}^n} |u|d\mu +\sqrt{\int_{\mathbb{R}^n} (\beta-\alpha)^2 d\mu}.
$$
Taking infimum in $u$ and repeating the argument we get inequality
$$
|\mathcal{N}(\alpha)-\mathcal{N}(\beta)|\leq \sqrt{\int_{\mathbb{R}^n} (\beta-\alpha)^2 d\mu}.
$$

{\it Step 5.} In \cite[Proof of Prop. 4.1]{Ambrosio} the following consequence of the Riesz--Markov--Kakutani representation theorem is derived: there exists an $\mathbb{R}^n-$valued Borel measure $\pi$ on $\mathbb{R}^n,$ such that 

\begin{enumerate}
\item
for all $f\in \mathcal{FC}^\infty$ one has
$$
\int_{\mathbb{R}^n} (Df,d\pi)=\int_{\mathbb{R}^n} f d(\nu_1-\nu_0);
$$

\item $W_1(\nu_0,\nu_1)$ coincides with the total variation of $\pi:$
$$
W_1(\nu_0,\nu_1)=\|\pi\|_v=\bigg(\sum^n_{i=1}\|\pi_i\|^2_v\bigg)^{\frac{1}{2}}.
$$

\end{enumerate}

From the symmetry of the Ornstein-Uhlenbeck semigroup the following relations follow.
$$
\int_{\mathbb{R}^n} f T_t\frac{d(\nu_1-\nu_0)}{d\mu}d\mu=\int_{\mathbb{R}^n} T_tf d(\nu_1-\nu_0)=
$$
$$
=\int_{\mathbb{R}^n} (DT_tf,d\pi)=e^{-t}\int_{\mathbb{R}^n} (T_tDf,d\pi)=\int_{\mathbb{R}^n} \bigg(Df,e^{-t}\frac{dT_t\pi}{d\mu}\bigg)d\mu.
$$
In other words,
$$
I\bigg(e^{-t}\frac{dT_t\pi}{d\mu}\bigg)=T_t\frac{d(\nu_1-\nu_0)}{d\mu}.
$$
In particular,
$$
\inf_{Iu=T_t\frac{d(\nu_1-\nu_0)}{d\mu}}\bigg\{\int_X |u(x)|\mu(dx) \bigg\}\leq \int_{\mathbb{R}^n} \bigg|e^{-t}\frac{dT_t\pi}{d\mu}\bigg|\mu(dx) \leq W_1(\nu_0,\nu_1).
$$
From the continuity proved at the step 4, the left-hand side in the preceeding inequality converges to the  
$$
\inf_{Iu=\frac{d(\nu_1-\nu_0)}{d\mu}}\bigg\{\int_X |u(x)|\mu(dx) \bigg\}.
$$
This gives the inequality $\geq$ in \eqref{eq_main} and finishes the proof.
\end{proof}

\section{Infinite Dimensional Case. Proof of the theorem \ref{thm_main}}

In this section we give the proof of the theorem \ref{thm_main} in the infinite dimensional case via the reduction to the finite dimensional case considered in the previous section. The same considerations (without limitting procedure) work when $X$  is finite dimensional but $\mu$ is not necessarily standard Gaussian measure. 

Throughout the section we assume that $X$ is an infinite dimensional separable Banach space, $\mu$ is a centered Gaussian measure on $X$ with $\mbox{supp}\mu=X.$ Then the Cameron-Martin space $H$ is an infinite-dimensional separable Hilbert space, which is densely and continuously embedded in $X.$ In particular, there is an orthonormal basis $\{e_n\}_{n\geq 1}$ in $H,$ such that $e_n\in X^*.$ 

\begin{proof}

Introduce ``projection operators'' $P_n:X\to \mathbb{R}^n,$
$$
P_n(x)=(e_1(x),\ldots,e_n(x)), \ x\in X,
$$
and finite-dimensional ``approximations'' 
$$
\mu^{(n)}=\mu\circ P^{-1}_n, \ \nu^{(n)}_i=\nu_i\circ P^{-1}_n, i=0,1.
$$
Following the notation from the section 3, we will denote the right-hand side of 
\eqref{eq_main} by $\mathcal{N}:$
$$
\mathcal{N}(\alpha)=\inf_{u\in L^2(X,\mu;H), Iu=\alpha}\bigg\{\int_{X} |u| d\mu \bigg\},
$$
$$
\alpha \in L^2(X,\mu), \int_X \alpha d\mu=0.  
$$
Analogous quantities for finite-dimensional spaces $\mathbb{R}^n$ will be denoted by $\mathcal{N}^{(n)}:$
$$
\mathcal{N}^{(n)}(\alpha)=\inf_{u\in L^2(\mathbb{R}^n,\mu^{(n)};\mathbb{R}^{n}), Iu=\alpha}\bigg\{\int_{\mathbb{R}^n} |u| d\mu^{(n)} \bigg\},
$$
$$
\alpha \in L^2(\mathbb{R}^n,\mu^{(n)}), \int_{\mathbb{R}^n} \alpha d\mu^{(n)}=0. 
$$

{\it Step 1.} Denote $\alpha=\frac{d(\nu_1-\nu_0)}{d\mu}.$ Let $\alpha_n$ be the conditional expectation of $\alpha$ with respect to $P_n:$
\begin{equation}
\label{eq_44}
\alpha_n=\mathbb{E}_\mu[\alpha|P_n].
\end{equation}
Then 
$$
\alpha_n(x)=\alpha^{(n)}(P_n(x)),
$$
where 
$$
\alpha^{(n)} \in L^2(\mathbb{R}^n,\mu^{(n)}), \int_{\mathbb{R}^n} \alpha^{(n)} d\mu^{(n)}=0. 
$$
At this step we prove that 
\begin{equation}
\label{eq_41}
\mathcal{N}^{(n)}(\alpha^{(n)})=\mathcal{N}(\alpha_n).
\end{equation}
As a consequence, using results of the step 4 of the proof of lemma \ref{lem_finite_dim}, we obtain the convergence
$$
\mathcal{N}^{(n)}(\alpha^{(n)})\to \mathcal{N}(\alpha), \ n\to \infty.
$$

To check \eqref{eq_41}, consider $u\in L^2(X,\mu;H),$ such that $Iu=\alpha_n.$ Define $u^{(n)}\in L^2(\mathbb{R}^n,\mu^{(n)};\mathbb{R}^{(n)})$ as follows
$$
\mathbb{E}_\mu[(u,e_i)|P_n](x)=u^{(n)}_i(P_n(x)), \ i=1,\ldots,n.
$$
Then, in the space $L^2(\mathbb{R}^n,\mu^{(n)}),$ one has
\begin{equation}
\label{eq_42}
Iu^{(n)}=\alpha^{(n)}.
\end{equation}
Indeed, consider a function $\beta^{(n)}\in\mathcal{FC}^\infty$ and a corresponding function 
$$
\beta(x)=\beta^{(n)}(P_n(x)).
$$
Following relations follow from the inclusion  $D\beta(x)\in \mbox{span}(e_1,\ldots,e_n)$ and the fact that $D\beta$ is $\sigma(P_n)-$measurable:
$$
\int_{\mathbb{R}^n} \beta^{(n)} \alpha^{(n)} d\mu^{(n)}=
\int_X \beta \alpha_n d\mu =\int_X \beta Iu d\mu =
$$
$$
=\int_X (D\beta,u)d\mu= \int_X \bigg(D\beta, \sum^n_{i=1} (u,e_i)e_i\bigg)d\mu=\int_{\mathbb{R}^n} (D\beta^{(n)},u^{(n)})d\mu^{(n)}.
$$
Equation \eqref{eq_42} implies that
$$
\mathcal{N}^{(n)}(\alpha^{(n)})\leq \int_{\mathbb{R}^n} |u^{(n)}|d\mu^{(n)}\leq \int_X |u|d\mu.
$$
This proves $\leq$ in \eqref{eq_41}. To check the reversed inequality, it is enough to consider $u^{(n)}\in L^2(\mathbb{R}^n,\mu^{(n)};\mathbb{R}^{(n)}),$ such that $Iu^{(n)}=\alpha^{(n)},$ define $u\in L^2(X,\mu;H)$ as
$$
u(x)=\sum^n_{i=1}u^{(n)}_i(P_n(x))e_i,
$$
and repeat previous considerations.

{\it Step 2.} From the lemma \ref{lem_finite_dim} the convergence
$$
W_1(\nu^{(n)}_0,\nu^{(n)}_1)=\mathcal{N}^{(n)}(\alpha^{(n)})\to \mathcal{N}(\alpha), \ n\to\infty
$$
follows.

Observe that
\begin{equation}
\label{eq_43}
W_1(\nu^{(n)}_0,\nu^{(n)}_1)\leq W_1(\nu_0,\nu_1).
\end{equation}
Indeed, for any bounded $1-$Lipshitz function $f^{(n)}:\mathbb{R}^n\to\mathbb{R},$ the function
$$
f:X\to \mathbb{R}, f(x)=f^{(n)}(P_n(x)),
$$
is $1-$ Lipshitz function on $X$ with respect to the distance $d_H(x,y)=|x-y|.$  By the Kantorovich--Rubinstein theorem,
$$
\int_{\mathbb{R}^n}f^{(n)} d(\nu^{(n)}_1-\nu^{(n)}_0)=\int_X f d(\nu_1-\nu_0)\leq W_1(\nu_0,\nu_1).
$$
Then inequality \eqref{eq_43} follows after taking the supremum in $f^{(n)}$  and applying the Kantorovich-Rubinstein theorem again. The inequality
$$
\mathcal{N}(\alpha)\leq W_1(\nu_0,\nu_1) 
$$
is proved.

To prove the revesed inequality, consider a bounded $1-$Lipschitz function $f:X\to \mathbb{R}$ (relatively to the distance $d_H(x,y)=|x-y|$). From the assumption 
$$
\alpha=\frac{d(\nu_1-\nu_0)}{d\mu}\in L^2(X,\mu)
$$ and the definition of $\alpha_n$ \eqref{eq_44} it follows that 
$$
\int_X f d(\nu_1-\nu_0) = \lim_{n\to \infty} \int_X \mathbb{E}_\mu[f|P_n] \alpha_n d\mu.
$$
Observe that 
$$
\mathbb{E}_\mu[f|P_n](x)=f^{(n)}(P_n(x)),
$$
where $f^{(n)}$ is a bounded $1-$Lipschitz function on $\mathbb{R}^n.$ Then
$$
\int_X f d(\nu_1-\nu_0) = \lim_{n\to \infty}\int_{\mathbb{R}^n} f^{(n)} d(\nu^{(n)}_1-\nu^{(n)}_0)\leq \lim_{n\to\infty} W_1(\nu^{(n)}_0,\nu^{(n)}_1)=\mathcal{N}(\alpha).
$$
It remains to take the supremum in $f$  and applying the Kantorovich-Rubinstein theorem to obtain reversed inequality
$$
W_1(\nu_0,\nu_1) \leq \mathcal{N}(\alpha).
$$
The theorem is proved.

\end{proof}

\end{document}